\documentclass[12pt,reqno]{amsart}
\usepackage{amsmath,amssymb}
\usepackage[english]{babel}
\usepackage{hyperref}
\usepackage{cleveref}
\usepackage{caption}
\usepackage{amssymb,amsmath,euscript,enumerate,tikz}
\usepackage[margin=1in]{geometry}
\usepackage{graphicx}
\usepackage{float}
\usepackage{pgf,tikz}
\usepackage{mathrsfs}
\usepackage{upgreek}
\usepackage[normalem]{ulem}

\newcommand \reg{\operatorname{reg}}
\newcommand \dist{\operatorname{dist}}
\newcommand \cone{\operatorname{cone}}

\newcommand \Tor{\operatorname{Tor}}
\newcommand \soc{\operatorname{Soc}}

\newcommand \pd{\operatorname{pdim}}

\newcommand \ini{\operatorname{in}}

\newcommand \K{\mathbb{K}}

\newtheorem{theorem}{Theorem}[section]
\newtheorem{definition}[theorem]{Definition}
\newtheorem{lemma}[theorem]{Lemma}
\newtheorem{proposition}[theorem]{Proposition}
\newtheorem{example}[theorem]{Example}

\newtheorem{question}[theorem]{Question}

\newtheorem{corollary}[theorem]{Corollary}

\newtheorem{remark}[theorem]{Remark}
\newtheorem{cons}[theorem]{Construction}

\begin{document}
\title[Level and pseudo-Gorenstein binomial edge ideals]{Level and pseudo-Gorenstein binomial edge ideals}

	\author[Giancarlo Rinaldo]{Giancarlo Rinaldo}
\email{giancarlo.rinaldo@unime.it}
\address{Department	of Mathematics, Informatics, Physics and Earth Science, University of Messina, Viale F. Stagno d’Alcontres, 31, Messina, 98166, Italy}
\author[Rajib Sarkar]{Rajib Sarkar}
\email{rajib@math.tifr.res.in/rajib.sarkar63@gmail.com}
\address{School of Mathematics, Tata Institute of Fundamental Research Bombay, Navy Nagar, Mumbai, 400005, India}

\begin{abstract}
We prove that level binomial edge ideals with regularity $2$ and pseudo-Gorenstein binomial edge ideals with regularity $3$ are cones, and we describe them completely.
Also, we characterize level and pseudo-Gorenstein binomial edge ideals of bipartite graphs. 
\end{abstract}
\maketitle

\section{Introduction}
Let $G$ be a simple graph with the vertex set $[n]$ and the edge set $E(G)$. The notion of binomial edge ideals was introduced by Herzog et al. in \cite{HHHKR}, and independently by Ohtani in \cite{oh}, see \cite{HHO-book} for nice survey on this topic. After that, many authors focused on the  problem
of finding a characterization of Cohen–Macaulay binomial edge ideals. There are several attempts to solve this problem for some families of graphs. Some papers in this direction are \cite{BMRS,DAV,DAV2,EHH-NMJ,AR2,LMRR-S_2,RR,Rinaldo-BMS,Rinaldo-Cactus,KM-JA,SS-22}
In particular, in \cite{DAV2},  the authors introduce nice combinatorial properties of Cohen-Macaulay binomial edge ideals that have been useful in \cite{LMRR-S_2} to give a computational classification of all the Cohen-Macaulay binomial edge ideals of graphs with at most $12$ vertices.

Nevertheless, as in the case of classical edge ideals, an exhaustive classiﬁcation of graphs whose binomial edge ideals are Cohen-Macaulay seems to be a hopeless task.

In the same time a fundamental algebraic invariant has been studied in deep that is the Castelnuovo-Mumford regularity of binomial edge ideal, using the combinatorics of the underlying graph.  This has been used to classify particular Cohen-Macaulay binomial edge ideals (see \cite{ERT-Licci} and \cite{KM-JA}).

Recently,  Gorenstein binomial edge ideals have been completely characterized  (see \cite{GBEI}). They are the ones defined on paths.

There are two interesting generalizations of Gorenstein rings: level rings (see \cite{Stanley}) and  pseudo-Gorenstein rings (see \cite{EHHM-pseudo}). Observing that  a ring  is Gorenstein if and only if the canonical module is a cyclic module, and hence generated in a single degree, the two generalizations naturally arise. In fact, if  one only
requires that the generators of the canonical module  are of the same degree, then the ring is called level,
and if one requires that there is only one generator of least degree, then we call it  pseudo-Gorenstein.

Two useful tools for the classification of Cohen-Macaulay binomial edge ideals are the constructions described in the paper \cite{RR}. As an example those constructions were used in deep to characterize the Cohen-Macaulay bipartite  graphs (see \cite{DAV}). The first one  is given by the union of two Cohen-Macaulay graphs along one of their simplicial vertices, and  the second one by cones from a vertex to two Cohen-Macaulay connected components. 

In the first section, we prove that the behavior of these two constructions extends naturally to level and pseudo-Gorenstein rings, giving us a way to generate infinite families of level and pseudo-Gorenstein rings by simple ones. Another  immediate consequence of the first construction is that block graphs are level, and within them, the only pseudo-Gorenstein are the Gorenstein ones. With respect to the cone construction, we observe that level cones need to have regularity $2$. This is a fundamental tool to classify all the level rings of regularity 2 that is one of the main aims of Section \ref{sec:regularity2}. In the meantime we describe a way to construct infinite families of pseudo-Gorenstein rings with fixed regularity.

In the paper \cite{KM-JA}, the binomial edge ideals with regularity $2$ is given. By this result and improving \cite[Proposition 3.4]{KM-JA} in Section \ref{sec:regularity2}, we classify the Cohen-Macaulay rings with regularity $2$ and prove that they are all level. A complete description of those graphs is given in Construction \ref{construction-CM-2}, that uses inductively cones. One can easily see that the only pseudo-Gorenstein rings of regularity 2 are the Gorenstein ones.

In general, level binomial edge ideals with regularity $3$ are not cones. An example is the bipartite graph $F_m$ (see Section \ref{sec:bipartite}). In Section \ref{sec:regularity3}, we classify all pseudo-Gorenstein rings of regularity $3$. These are the cones defined in Construction \ref{construction-pseudo-Gorenstein}. To obtain this result, we prove another one that is interesting by itself. Any $P_5$-free accessible graph (see \cite{DAV2}) is Cohen-Macaulay, giving a positive answer to the \cite[Conjecture 1.1]{DAV2}. Observe that there exist pseudo-Gorenstein rings of regularity $4$ that are not cones (see Section \ref{sec: computation}). 

In \cite{DAV}, a complete classification of Cohen-Macaulay bipartite binomial edge ideal is given due to two constructions. The first is related to the decomposability one. The second one is obtained by collapsing two leaves of two graphs and gluing them, see Definition \ref{def:dot}, namely the $\circ$-operation. The $\circ$-operation has different behavior with respect to level and pseudo-Gorenstein binomial edge ideal of bipartite graphs. In fact, by Theorem \ref{thm:level-F_m} $F_m$ is level, but using the $\circ$-operation on two or more $F_m$ graphs we obtain a non-zero Betti number, whose position is induced by the diameter and it is not the extremal one, see Lemma \ref{lemm:non-zero-betti-number}.
In the pseudo-Gorenstein case $F_m$ is pseudo-Gorenstein if and only if $m\in \{1,2\}$ that is, $F_m$ is a path. Moreover, the $\circ$-operation gives us an infinite family of pseudo-Gorenstein bipartite graphs that are not paths. In particular, these graphs are of the form either $F_3\circ F_3$ or $F_{m_1}\circ F_{m_2}\circ \cdots \circ F_{m_t}$, where $m_1=m_t=3$ and $m_i=4$ for $2\leq i\leq t-1$ with $t\geq 3$, see Theorem \ref{thm:pseduo-Gorenstein-bipartite}.

Another interesting structure related to our research is given by matroid simplicial complexes. In general (see \cite[Chapter III, Section 3]{Stanley}) if $\Delta$ is matroid, then the monomial ideal $I(\Delta)$ is level. Since a complete description of the initial ideal of $J_G$ and its simplicial complex has been given (\cite{LMRR-S_2}), in Section \ref{sec:matroid} we study under what condition this complex is a matroid one. We prove that the only one that is matroid is related to the path.

In the last section, by using a computational approach and the database of all Cohen-Macaulay binomial edge ideals on connected and indecomposable graphs with at most 12 vertices (see \cite{LMRR}), we provide a list of all level and pseudo-Gorenstein binomial edge ideals with at most 12 vertices that is downloadable from \cite{RS}.
In the table \ref{tab:indec}, we provide the cardinalities of these sets with respect to the number of vertices, too. Moreover, some interesting examples are examined.

\vskip 2mm
\noindent
\textbf{Acknowledgment:}
The second author thanks the University of Messina for the hospitality where some of the works have been done.

\section{Decomposability and Cone constructions}\label{sec:decomposablity}
First, we recall some definitions and notation of graphs that we will use in this article. Let $G$ be a simple graph with the vertex set $V(G)=[n]:= \{1, \ldots, n\}$ and edge set $E(G)$. For $i, j \in V(G),$ let $\dist_G(i,j)$ denote the smallest length of paths connecting $i$ and $j$ in $G$. Let $d(G) = \max\{\dist_G(i,j) ~:~ i, j \in [n]\}$ denote the \textit{diameter} of $G$. A subset $C\subseteq V(G)$ is said to be a \textit{clique} of $G$ if for all $i,j\in C$ with $i \neq j$ one has $\{i, j\} \in E(G)$. The \textit{clique complex} $\Delta(G)$ of $G$ is the simplicial complex of all its cliques. A clique $C$ of $G$ is called a \textit{face} of $\Delta(G)$ and its \textit{dimension} is given by $|C| -1$.  A vertex of $G$ is called a {\em simplicial vertex} of $G$ if it belongs to exactly one maximal clique of $G$. A vertex of $G$ is called a \textit{cut vertex} if the removal of the vertex increases the number of connected components. 
For each $i \in T$, if $i$ is a cut vertex of the induced subgraph on $([n]\setminus T)\cup \{i\}$,
then we say that $T$ has the cut point property. Set $\mathcal{M}(G) = \{ T: T \; \text{has the cut point property} \}$. A graph $G$ is {\em decomposable}, if there exist two subgraphs $G_1$ and $G_2$ of $G$,  and a decomposition $G=G_1\cup G_2$ with $V(G_1)\cap V(G_2)=\{v\}$, where $v$ is a simplicial vertex of $G_1$ and $G_2$. If $G$ is not decomposable, we call it \textit{indecomposable}.

Let $S=
\K[x_1, \ldots, x_{n}, y_1, \ldots, y_{n}]$ be the polynomial ring over an arbitrary field $\K$. The \textit{binomial edge ideal} of $G$ is defined as $$J_G:=(
x_i y_j - x_j y_i ~ : i < j \text{ and } \{i,j\}\in E(G)) \subseteq S.$$

Let $M$ be a finitely graded $S$-module and $\beta_{i,i+j}(M)$ the graded Betti numbers. The 
\textit{projective dimension} of $M$ is defined as $\pd(M):=\max\{i : \beta_{i,i+j}(M) \neq 0 \text{ for some } j\}$
and the \textit{Castelnuovo-Mumford regularity} (or simply, \textit{regularity}) of $M$ is defined as 
$\reg(M):=\max \{j : \beta_{i,i+j}(M) \neq 0 \text{ for some } i\}.$ If $\beta_{i,i+j}(M)\neq 0$ and for all pairs $(k,l)\neq (i,j)$ with $k\geq i$ and $l\geq j$, $\beta_{k,k+l}(M)=0$, then $\beta_{i,i+j}(M)$ is called an \textit{extremal Betti number} of $M$. When $M$ has an unique extremal Betti number, we denote it by $\hat{\beta}(M)$. We recall that if $M$ is Cohen-Macaulay, then it admits the unique extremal Betti number that is $\hat{\beta}(M)=\beta_{p,p+r}(M)$, where $p=\pd(M)$ and $r=\reg(M)$.

In \cite[Proposition 3]{HR-EJC},
Herzog and Rinaldo obtain the formula for the Betti polynomial of a decomposable graph in terms of the Betti polynomial of its indecomposable subgraphs. With the help of this result, we prove the levelness and pseudo-Gorensteinness on decomposable graphs.
\begin{proposition}\label{prop:decomposability-level-pseudo}
	Let $G$  be a decomposable  graph with decomposition $G=G_1\cup G_2$. Then $S/J_G$ is level (respectively, pseudo-Gorenstein) if and only if $S_{1}/J_{G_1}$ and $S_{2}/J_{G_2}$ are level (respectively, pseudo-Gorenstein), where $S_{i}=\K[x_j,y_j:j\in V(G_i)]$ for $i=1,2$.
\end{proposition}
\begin{proof}
	Set $|V(G_1)|=n_1$ and $|V(G_2)|=n_2$. It follows from \cite[Theorem 2.7]{RR} that $S/J_G$ is Cohen-Macaulay if and only if $S_{1}/J_{G_1}$ and $S_{2}/J_{G_2}$ are Cohen-Macaulay. Suppose now they are all Cohen-Macaulay. Then by \cite[Corollary 3.4]{HHHKR} and Auslander-Buchsbaum formula, $\pd(S_{1}/J_{G_1})=n_1-1$, $\pd(S_{2}/J_{G_2})=n_2-1$ and $\pd(S/J_G)=n_1+n_2-2$. Let $\reg(S_{1}/J_{G_1})=r_1$ and $\reg(S_{2}/J_{G_2})=r_2$. So, by \cite[Theorem 3.1]{JNR1}, $\reg(S/J_G)=r=r_1+r_2$. Hence, $\hat{\beta}(S_1/J_{G_1})=\beta_{n_1-1,n_1-1+r_1}(S_1/J_{G_1})$, $\hat{\beta}(S_2/J_{G_2})=\beta_{n_2-1,n_2-1+r_2}(S_2/J_{G_2})$ and $\hat{\beta}(S/J_{G})=\beta_{n_1+n_2-2,n_1+n_2-2+r}(S/J_{G})$. Assume that $S_{1}/J_{G_1}$ and $S_{2}/J_{G_2}$ are level. Therefore, $\beta_{n_1-1,n_1-1+i}(S_{1}/J_{G_1})=\beta_{n_2-1,n_2-1+j}(S_{2}/J_{G_2})=0$ for $1\leq i\leq r_1-1$ and $1\leq j\leq r_2-1$. By \cite[Proposition 3]{HR-EJC}, $\hat{\beta}(S/J_{G})=\hat{\beta}(S_1/J_{G_1})\hat{\beta}(S_2/J_{G_2})$ and 
	
	$$\beta_{n_1+n_2-2,n_1+n_2-2+a}(S/J_G)=\underset{{\substack{i_1+i_2=n_1+n_2-2 \\ j_1+j_2=a}}}{\sum} \beta_{i_1,i_1+j_1}(S_{1}/J_{G_1})\beta_{i_2,i_2+j_2}(S_{2}/J_{G_2}) \text{ for } 1\leq a \leq r-1.$$
	
	Since $i_1\leq n_1-1$, $i_2\leq n_2-1$ and	$n_1+n_2-2=i_1+i_2$, we have $i_1=n_1-1$ and $i_2=n_2-1$. Now $a=j_1+j_2\leq r_1+r_2-1$ with $j_1\leq r_1$ and $j_2\leq r_2$ gives that $j_1\leq r_1-1$ or $j_2\leq r_2-1$. Therefore, $\beta_{i_1,i_1+j_1}(S_{1}/J_{G_1})=0$ or $\beta_{i_2,i_2+j_2}(S_{2}/J_{G_2})=0$, and so $\beta_{n_1+n_2-2,n_1+n_2-2+a}(S/J_G)=0$ for $1\leq a\leq r-1$.
	Hence, $S/J_G$ is level.
	
	Suppose now $S/J_G$ is level. If $S_{1}/J_{G_1}$ is not level, then there exists $1\leq i\leq r_1-1$ such that $\beta_{n_1-1,n_1-1+i}(S_{1}/J_{G_1})\neq 0$. This implies that 
	$$\beta_{n_1-1,n_1-1+i}(S_{1}/J_{G_1})\beta_{n_2-1,n_2-1+r_2}(S_{2}/J_{G_2})=\beta_{n_1+n_2-2,n_1+n_2-2+i+r_2}(S/J_G)\neq 0,$$
	which is a contradiction. Similarly, we can arrive at a contradiction if $S_{2}/J_{G_2}$ is not level. Hence, $S_{1}/J_{G_1}$ and $S_{2}/J_{G_2}$ are level.
	
	For the pseudo-Gorenstein part, by \cite[Corollary 4]{HR-EJC}, the unique extremal Betti number of $S/J_G$ is given by $\hat{\beta}(S/J_G)=\hat{\beta}(S_{1}/J_{G_1})\hat{\beta}(S_{2}/J_{G_2})$. Therefore, $\hat{\beta}(S/J_G)=1$ if and only if $\hat{\beta}(S_{1}/J_{G_1})=\hat{\beta}(S_{2}/J_{G_2})=1$. Hence, $S/J_G$ is pseudo-Gorenstein if and only if $S_{i}/J_{G_i}$ is pseudo-Gorenstein for $i=1,2$.
\end{proof}

\begin{corollary}\label{cor:level-pseudo-decomposable}
	Let $G$ be a decomposable graph into indecomposable subgraphs $G_1,\ldots,G_r$. Then $S/J_G$ is a level ring (respectively, pseudo-Gorenstein) if and only if $S_i/J_{G_i}$ is a level ring (respectively, pseudo-Gorenstein) for all $1\leq i\leq r$, where $S_{i}=\K[x_j,y_j:j\in V(G_i)]$.
\end{corollary}
In conclusion, we get the following result.
\begin{corollary}\label{cor:CM-block-graph}
	Let $G$ be a Cohen-Macaulay block graph. Then $S/J_G$ is a level ring. Moreover, if $S/J_G$ is pseudo-Gorenstein, then $G$ must be a path.
\end{corollary}
\begin{proof}
	Note that a Cohen-Macaulay block graph is a decomposable graph where indecomposable subgraphs are complete graphs. Since binomial edge ideals of complete graphs are level, by Corollary \ref{cor:level-pseudo-decomposable}, $S/J_G$ is also level. 
	
	Let $S/J_G$ be pseudo-Gorenstein. Since $\hat{\beta}(S/J_G)=1$, it follows from \cite[Theorem 6]{HR-EJC} that each indecomposable subgraph of $G$ must be $K_2$, and hence $G=P_n$.
\end{proof}

Now we study the levelness and pseudo-Gorensteinness for the cone graph. Let $v\notin V(G)$. Then the \textit{cone} of $v$ on $G$ is the graph with the vertex set $V(G)\cup \{v\}$ and the edge set $E(G)\cup \{\{u,v\}:u\in V(G)\}$, and it is denoted by $\cone(v,G)$. Let $G=\cone(v,H_1\sqcup H_2)$ on $n$ vertices with $H_1$ and $H_2$ connected graphs. Then it is proved that if $J_{H_1}$ and $J_{H_2}$ are Cohen-Macaulay, then $J_G$ is also Cohen-Macaulay, \cite[Theorem 3.8]{RR}. The converse has recently been proved in \cite[Theorem 4.8]{DAV2}.
\begin{proposition}\label{cone-level-pseudo}
	Let $G$ be the $\cone(v,H_1\sqcup H_2)$, where $H_1$ and $H_2$ are connected Cohen-Macaulay graphs.  Then
	\begin{enumerate}
		\item $S/J_G$ is level if and only if $\reg (S/J_G)=2$.
		\item if $\reg(S/J_G)>2$, then $S/J_G$ is pseudo-Gorenstein if and only if $S_{H_i}/J_{H_i}$ is pseudo-Gorenstein for $i=1,2$.
	\end{enumerate}
\end{proposition} 
\begin{proof}
	$(1).$ If $\reg(S/J_G)>2$, then thanks to \cite[Proposition 3.2]{RM-Ext}, we have $\beta_{p,p+2}(S/J_G)=n-2,$ where $p=\pd(S/J_G)$. Therefore, $S/J_G$ is not level. Hence, if $S/J_G$ is level, then $\reg(S/J_G)=2$. Conversely, let $\reg(S/J_G)=2$. It follows from \cite[Corollary 4.3]{HKS} that \begin{equation}\label{linear-strand}
		\beta_{i,i+1} (S/J_G) = i f_{i}(\Delta(G)),
	\end{equation}
	where $\Delta(G)$ is the clique complex of $G$ and $f_{i}(\Delta(G))$ is the number of faces of $\Delta(G)$ of dimension $i$. Therefore, $\beta_{n-1,n}(S/J_{G})=0$ as $G$ is not complete. Hence $S/J_G$ is level.
	
	$(2).$ 
	It follows from \cite[Proposition 3.2]{RM-Ext} that if $\reg(S/J_G)> 2$, then $\hat{\beta}(S/J_G)=1$ if and only if $\hat{\beta}(S_{H_i}/J_{H_i})=1$ for $i=1,2$. Hence, $S/J_G$ is pseudo-Gorenstein if and only if $S_{H_i}/J_{H_i}$ is pseudo-Gorenstein for $i=1,2$.
\end{proof}

In the next section, we show that if $S/J_G$ is level with $\reg(S/J_G)=2$, then $G$ necessarily be a cone, and we characterize all such graphs in Corollary \ref{cor:level-reg=2}.

With the help of Proposition \ref{cone-level-pseudo}, we can construct an infinite family of pseudo-Gorenstein rings with arbitrary regularity greater than 2.
\begin{remark}
	Let $r>2$. For chosen $n_1$ and $n_2$, we can easily construct $H=P_{n_1}\sqcup P_{n_2}$ such that $\reg(S_H/J_H)=n_1+n_2-2=r$. Let $G=\cone(v,H)$. Then by \cite[Theorem 2.1]{KM-JA}, $\reg(S/J_G)=r$ and hence, it follows from Proposition \ref{cone-level-pseudo} that $S/J_G$ is pseudo-Gorenstein.
\end{remark}

It is natural to ask, under the condition that $S/J_G$ is pseudo-Gorenstein with $\reg(S/J_G)>2$, whether $G$ is a cone. For $\reg(S/J_G)=3$, we have an affirmative answer in Section \ref{sec:regularity3} (See Theorem \ref{thm:pseudo-Gorenstein-with-regularity3}). This is not true when the regularity is greater than $3$ by the examples on $7$ vertices provided in Section \ref{sec: computation}.

\section{Level with regularity $2$}\label{sec:regularity2}
For the aim of  the section we need the definition of the join product (or simply join) of two graphs. Let $G_i$ be a graph with the vertex set $[n_i]$ for $i=1,2$. Then the \textit{join product} of $G_1$ and $G_2$, denoted by $G_1*G_2$, is the graph with the vertex set $[n_1]\sqcup [n_2]$ and the edge set $E(G_1*G_2)=E(G_1)\cup E(G_2)\cup \{\{i,j\}:i\in [n_1] \text{ and } j\in [n_2]\}$. Observe that the cone of $v$ on $G$ is the join graph $v*G$. For a graph $G$ on $[n]$, $G$ is called a \textit{$k$-vertex-connected} if $k<n$ and for any subset $A\subseteq [n]$ with $|A|<k$, the induced subgraph on the vertex set $[n]\setminus A$ is connected.

We recall the following result due to Saeedi Madani and Kiani, \cite[Theorem 3.2]{KM-JA}
\begin{theorem}\label{KM-main-thm}
	Let $G$ be a non-complete graph with $n$ vertices and no isolated vertices.
	Then $\reg(S/J_G )=2$ if and only if either
	\begin{enumerate}
		\item $G=K_r \sqcup K_s$ for $r,s\geq 2$ and $r+s=n$, or
		\item $G=G_1*G_2$, where $G_i$ is a graph with $n_i<n$ vertices such that $n_1+n_2=n$ and $\reg(S_{G_i}/J_{G_i})\leq 2$ for $i=1,2$.
	\end{enumerate}
\end{theorem}

First, we construct some graphs in an inductive way:
\begin{cons}\label{construction-CM-2}
	Let $H_0=K_r\sqcup K_s$ with $1\leq r\leq s$. Let $H_1=K_1*H_0$. For $i\geq 2$, let $H_i=K_1*(K_1\sqcup H_{i-1})$ for $i\geq 2$.
\end{cons}
Let us illustrate our construction with some examples:
\begin{example}\label{exa:levelcones}
	We use the notation that the edges $K_{r+1}$ and $K_{s+1}$ of the graph $H_1$ in Figure 1 represent the corresponding complete graphs that have a unique vertex in common. 
	\begin{figure}[H]
		\centering
		\begin{tikzpicture}[scale=1]
			\draw  (-4,1)-- (-4,0);
			\draw  (-4,0)-- (-4,-1);
			\draw  (-2,0)-- (-1,0);
			\draw [dotted]  (-1,0)-- (0,1);
			\draw [dotted] (-1,0)-- (0,0);
			\draw [dotted] (-1,0)-- (0,-1);
			\draw  (0,1)-- (0,0);
			\draw  (0,0)-- (0,-1);
			\draw  (2,0)-- (3,0);
			\draw  (3,0)-- (4,1);
			\draw  (3,0)-- (4,0);
			\draw  (3,0)-- (4,-1);
			\draw  (4,1)-- (4,0);
			\draw  (4,0)-- (4,-1);
			\draw  (3,1)-- (2,1);
			\draw [dotted] (3,1)-- (2,0);
			\draw [dotted] (3,1)-- (3,0);
			\draw [dotted] (3,1)-- (4,1);
			\draw [dotted] (3,1)-- (4,0);
			\draw [dotted] (3,1)-- (4,-1);
			\begin{scriptsize}
				\fill  (-4,1) circle (2.5pt);
				\fill  (-4,0) circle (2.5pt);
				\fill  (-4,-1) circle (2.5pt);
				\fill  (-2,0) circle (2.5pt);
				\fill  (-1,0) circle (2.5pt);
				\fill  (0,1) circle (2.5pt);
				\fill  (0,0) circle (2.5pt);
				\fill  (0,-1) circle (2.5pt);
				\fill  (2,0) circle (2.5pt);
				\fill (3,0) circle (2.5pt);
				\fill  (4,1) circle (2.5pt);
				\fill  (4,0) circle (2.5pt);
				\fill  (4,-1) circle (2.5pt);
				\fill  (3,1) circle (2.5pt);
				\fill (3.16,1.43) node {$v_2$};
				\node at (2.16,1.43) {$u_2$};
				\fill (2,1) circle (2.5pt);
				\node at (-1,0.4) {$v_1$};
				\node at (-2,0.4) {$u_1$};
				\node at (-4.4,0.5) {$K_{r+1}$};
				\node at (-4.4,-0.5) {$K_{s+1}$};
				\node at (0.4,0.5) {$K_{r+1}$};
				\node at (0.4,-0.5) {$K_{s+1}$};
				\node at (4.4,0.5) {$K_{r+1}$};
				\node at (4.4,-0.5) {$K_{s+1}$};
				\node at (-4,-1.5) {$H_1$};
				\node at (-1,-1.5) {$H_2=v_1*(u_1 \sqcup H_1)$};
				\node at (3,-1.5) {$H_3=v_2*(u_2\sqcup H_2)$};
				
			\end{scriptsize}
		\end{tikzpicture}
		\caption{Level graphs of regularity $2$.}
	\end{figure}
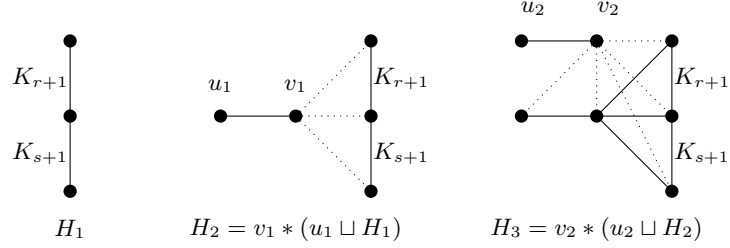
	
\end{example}

We need the following lemma in the proof of Theorem \ref{CM-reg=2}.
\begin{lemma}\label{tech-lemma}
	Let $G=K_1*(K_1\sqcup G')$ with $|V(G)|=n$. Then $J_G$ is Cohen-Macaulay if and only if $J_{G'}$ is a Cohen-Macaulay.
\end{lemma}
\begin{proof}
	It follows from \cite[Theorem 3.8]{RR} that if $J_{G'}$ is Cohen-Macaulay, then $J_G$ is Cohen-Macaulay. Also, if $J_G$ is Cohen-Macaulay, then by \cite[Theorem 4.8(4)]{DAV2}, we get that $J_{G'}$ is Cohen-Macaulay.
\end{proof}

Now we characterize Cohen-Macaulay $J_G$ with $\reg(S/J_G)=2$. This is an improvement of the result of \cite[Proposition 3.4]{KM-JA}. 
\begin{theorem}\label{CM-reg=2}
	Let $G$ be a connected graph on the vertex set $[n]$. Then the following conditions are equivalent:
	\begin{enumerate}
		\item $S/J_G$ is Cohen–Macaulay and $\reg(S/J_G)=2$;
		\item $G=H_i$ for some $i\geq 1$, where $H_i$'s are defined in Construction \ref{construction-CM-2}.
	\end{enumerate}
\end{theorem}
\begin{proof}
	$(2)\implies (1)$
	Note that $H_1$ is a decomposable graph with decomposition $H_1=K_{r+1}\cup K_{s+1}$. By \cite[Theorem 1.1]{EHH-NMJ}, $J_{H_1}$ is Cohen-Macaulay, and hence by applying repeatedly Lemma \ref{tech-lemma}, we can conclude that $J_{H_i}$ is Cohen-Macaulay for $i\geq 2$. Since for any complete graph $K_t$, by \cite[Theorem 2.1]{KM-EJC} $\reg(S_{K_t}/J_{K_t})=1$. Therefore, it follows from \cite[Theorem 3.1]{JNR1} that $\reg(S_{H_1}/J_{H_1})=2$. Hence if $G=H_i$ for $i\geq 2$, then by \cite[Theorem 2.1]{KM-JA}, $\reg(S/J_G)=2$.
	
	$(1)\implies (2)$ 
	Let $G$ be a connected non-complete graph. Then by Theorem \ref{KM-main-thm}, $G$ can be written as $G=G_1*G_2$ where $G_i$ is on $n_i$ vertices and $\reg(S_{G_i}/J_{G_i})\leq 2$ for $i=1,2$ such that $n_1+n_2=n$. We may assume that $n_1\leq n_2$. Then $G$ is $n_1$-vertex-connected and hence by \cite[Proposition 3.10]{BN17}, $n_1=1$ as $S/J_G$ is Cohen-Macaulay. Moreover, if $G_2$ is connected, then $G=K_1*G_2$ is $2$-vertex connected graph. Therefore, $G_2$ is disconnected with $n_2=n-1\geq 2$. We now proceed by induction on $n$. Suppose now $n=3$. Then $n_2=2$ and so $G_2$ is a union of two isolated vertices. Hence, from our construction, we have $G=H_1$ with $r=s=1$. Assume that $n\geq 4$. By \cite[Theorem 2.1]{KM-JA}, $\reg(S_{G_2}/J_{G_2})\leq 2$. First assume that $\reg(S_{G_2}/J_{G_2})=1$. Then $G_2$ can be written as $G_2=K_r^c\sqcup K_s$ for $r\geq 1,s\geq 2$, where $K_r^c$ represents a set of $r$ isolated vertices. This implies that $G$ is a block graph where $r$-copies of $K_2$ and $K_{s+1}$ intersect at $K_1$. Therefore, thanks to \cite[Theorem 1.1]{EHH-NMJ}, we get $r=1$, and hence $G=H_1$ with $r=1,s\geq 2$. 
	
	Suppose now $\reg(S_{G_2}/J_{G_2})=2$. Then $G_2$ can have at most two non-trivial connected components. If $G_2$ has two non-trivial connected components, then we can write $G_2=K_r\sqcup K_s\sqcup K_t^c$ with $2\leq r\leq s$ and $t\geq 0$. Then similarly, $G$ is a block graph where $K_{r+1},K_{s+1}$ and $t$-copies of $K_2$ intersect at $K_1$, and hence $t=0$. So, $G=K_1*(K_r\sqcup K_S)=H_1$ with $2\leq r\leq s$. 
	
	Assume that $G_2$ has only one non-trivial connected component, say $G_2'$. Since $G_2$ is disconnected and $S/J_G$ is Cohen-Macaulay, $G_2$ has one isolated vertex. Then by Lemma \ref{tech-lemma}, $S_{G_2'}/J_{G_2'}$ is Cohen-Macaulay. So, by inductive hypothesis, we obtain $G_2'=H_i$ for some $i$, and hence, $G=H_{i+1}$.
	This completes the proof.
\end{proof}

\begin{corollary}\label{cor:level-reg=2}
	Let $G$ be a connected graph on the vertex set $[n]$. Then the following conditions are equivalent:
	\begin{enumerate}
		\item $S/J_G$ is level and $\reg(S/J_G)=2$;
		\item $G=H_i$ for some $i\geq 1$, where $H_i$'s are defined in Construction \ref{construction-CM-2}.
	\end{enumerate}
\end{corollary}
\begin{proof}
	$(1)\implies (2)$ It follows from Theorem \ref{CM-reg=2}.	$(2)\implies (1)$ If $G=H_i$, then by Theorem \ref{CM-reg=2}, $S/J_G$ is Cohen-Macaulay. Moreover, by the equation \eqref{linear-strand}, $\beta_{p,p+1}(S/J_G)=0$, where $p=\pd(S/J_G)$. Since $\reg(S/J_G)=2$, $S/J_G$ is level. 
\end{proof}

\begin{remark}\label{rmk:reg<=2}
	Let $G$ be a Cohen-Macaulay graph with $\reg(S/J_G)\leq 2$. Then by Corollaries \ref{cor:CM-block-graph} and \ref{cor:level-reg=2}, $S/J_G$ is level. Hence, $S/J_G$ is pseudo-Gorenstein if and only if $G$ is a path.
\end{remark}
\section{Pseudo-Gorenstein with regularity $3$}\label{sec:regularity3}
This Section studies the pseudo-Gorenstein binomial edge ideal $J_G$ with $\reg(S/J_G)=3$. Let us recall some definitions which will be useful in this section. If $G$ has no induced subgraph isomorphic to the path with $k$ vertices, namely $P_k$, then $G$ is said to be a \textit{$P_k$-free graph}. A subset $X\subseteq V(G)$ is called a \textit{dominating set of G} if for every vertex $v\in V(G)$, either $v\in X$ or there exists $u\in X$ such that $\{u,v\}\in E(G)$. By a \textit{connected dominating} set $X$ we mean that $X$ is a dominating set of $G$ and the induced subgraph on $X$, $G_X$ is connected. A connected
dominating set such that every proper subset is not a connected dominating set is said to be a \textit{ minimal connected dominating set}. A connected dominating set of
minimum size is called a \textit{ minimum connected dominating set} (see also \cite{KM-JA}). 
\begin{theorem}$($\cite[Theorem 4]{CS-P_k-free}$)$\label{p_k-free}
	Let $G$ be a connected $P_k$-free graph for $k>3$ and $X$ be any minimum connected dominating set of $G$. Then $G_X$ is $P_{k-2}$-free or isomorphic to $P_{k-2}$.
\end{theorem}
\begin{lemma}\label{cut-vertices-with-dominating-set}
	Let $G$ be a connected graph with $r$ cut vertices, $v_1,\dots,v_r$. Then for every connected dominating set $X$ of $G$, $\{v_1,\dots,v_r\}\subseteq X$. Moreover, if $G$ is accessible, then $G$ has only one minimum connected dominating set that is,  $X=\{v_1,\dots,v_r\}$.
\end{lemma}
\begin{proof}
	Let $X$ be a connected dominating set and $v_i$ a cut vertex with $v_i\notin X$. Let $G=G_1\cup G_2$ with $V(G_1)\cap V(G_2)=\{v_i\}$ and  let $u_1\in V(G_1)\cap X$,  $u_2\in V(G_2)\cap X$ . Since $G_X$ is a connected graph there exists a path from $u_1$ to $u_2$ and this path must pass through $v_i$. Contradiction.
	
	If $G$ is accessible, then it follows from \cite[Lemma 4.9 and Theorem 4.12]{DAV2} that each vertex is adjacent to  a cut vertex of $G$ and $G_X$ is connected  for  $X=\{v_1,\dots,v_r\}$. Therefore $X$ is the minimum connected dominating set by the previous part of the proof.
\end{proof}

In \cite[Conjecture 1.1]{DAV2}, the author gave a conjecture about the characterization of Cohen-Macaulay graphs in terms of the structure of the underlying graphs. 
Here, we prove the conjecture under a given condition.
\begin{theorem}\label{p_5-free-conjecture}
	Let $G$ be a $P_5$-free graph. Then the following conditions are equivalent:
	\begin{enumerate}
		\item $J_G$ is strongly unmixed;
		\item $J_G$ is Cohen-Macaulay;
		\item $G$ is accessible.
	\end{enumerate}
\end{theorem}
\begin{proof}
	$(1)\implies (2)$ and $(2)\implies (3)$ follow from \cite[Theorem 5.11 and Theorem 3,5]{DAV2} respectively. So, we only need to prove $(3)\implies (1)$. Let $G$ be an accessible graph. Since $G$ is $P_5$-free, every induced path connecting two vertices of $G$ is of length less than $4$. Therefore, $\dist_G(u,v)\leq 3$ for any vertices $u,v$, and hence $d(G)\leq 3$. Thanks to \cite[Lemma 4.2]{JR2}, we get $d(G_v)=d(G_v\setminus v)\leq d(G) \leq 3$. Therefore, $G_v$ and $G_v\setminus v$ are $P_5$-free graphs. Moreover, any induced subgraph of $G$ is $P_5$-free. Hence $G\setminus v$ is $P_5$-free, too. By Theorem \ref{p_k-free}, $G_X$ is $P_3$-free or isomorphic to $P_3$, for any minimum connected dominating set $X$ of $G$. If $G_X$ is $P_3$-free then it is the complete graph $K_m$ for some $m\geq 1$. 
	We show that $G_X$ can not be isomorphic to $P_3$. In fact, if $G_X$ is isomorphic to $P_3$, then $X=\{v_1,v_2,v_3\}$ is the set of cut vertices of $G$ by Lemma \ref{cut-vertices-with-dominating-set}. Hence  $G$ has an induced path $P=(u,v_1,v_2,v_3,v)$ of length $4$ with $u\in N_G(v_1)$ and $v\in N_G(v_3)$. This contradicts the fact that $G$ is a $P_5$-free graph. Thus, $G_X$ is isomorphic to $K_m$ for some $m\geq 1$. Then by \cite[Proposition 6.6]{DAV2}, there exists a cut vertex $v\in X$ such that $J_{G\setminus v}$ is unmixed. Therefore, by \cite[Corollary 5.16]{DAV2}, $G\setminus v,G_v$ and $G_v\setminus v$ are accessible. Hence, by setting $\mathcal{G}$ to be the class of $P_5$-free accessible graphs, the assertion follows from \cite[Proposition 5.13]{DAV2}.
\end{proof}

\begin{cons}\label{construction-pseudo-Gorenstein}
	Let $G_0^1=P_1\sqcup P_4$ and $G_0^2=P_2\sqcup P_3$. Let $G_1^1=K_1*G_0^1$ and $G_1^2=K_1*G_0^2$. For $i\geq 2$, construct $G_i^1=K_1*(K_1\sqcup G_{i-1}^1)$ and $G_i^2=K_1*(K_1\sqcup G_{i-1}^2)$.
\end{cons}

\begin{example}\label{exa:pseudo-gorenstein-cones}
	In Figure 2, we present the graphs $G_2^1$ and $G_2^2$ of Construction \ref{construction-pseudo-Gorenstein}.
	\begin{figure}[H]
		\centering
		\begin{tikzpicture}[scale=1]
			\draw[dotted] (-4, 1) -- (-4,0);
			\draw[dotted] (-4, 1) -- (-3,1);
			\draw[dotted] (-4, 1) -- (-3,-1);
			\draw[dotted] (-4, 1) -- (-3,0.3);
			\draw[dotted] (-4, 1) -- (-3,-0.34);
			\draw[dotted] (-4, 1) -- (-5,0);
			\draw (-4, 1) -- (-5,1);
			\draw  (-4,0)-- (-3,1);
			\draw  (-3,0.3)-- (-3,1);
			\draw  (-3,0.3)-- (-3,-0.34);
			\draw  (-3,-0.34)-- (-3,-1);
			\draw  (-4,0)-- (-3,-1);
			\draw  (-4,0)-- (-3,0.3);
			\draw  (-4,0)-- (-3,-0.34);
			\draw  (-4,0)-- (-5,0);
			\draw  (1,0)-- (2,1);
			\draw  (1,0)-- (2,0);
			\draw  (1,0)-- (2,-1);
			\draw  (2,1)-- (2,0);
			\draw  (2,-1)-- (2,0);
			\draw  (0,0.34)-- (1,0);
			\draw  (0,0.34)-- (0,-0.38);
			\draw  (0,-0.38)-- (1,0);
			\draw (1, 1) -- (0,1);
			\draw[dotted] (1, 1) -- (1,0);
			\draw[dotted] (1, 1) -- (2,0);
			\draw[dotted] (1, 1) -- (2,1);
			\draw[dotted] (1, 1) -- (2,-1);
			\draw[dotted] (1, 1) -- (0,0.34);
			\draw[dotted] (1, 1) -- (0,-0.38);
			\draw[dotted] (1, 1) -- (2,0);
			\begin{scriptsize}
				\fill  (-4,1) circle (2.5pt);
				\node at (-3.84,1.43) {$v_1$};
				\fill  (-5,1) circle (2.5pt);
				\node at (-4.84,1.43) {$u_1$};
				\fill  (-4,0) circle (2.5pt);
				\fill  (-5,0) circle (2.5pt);
				\fill  (-3,1) circle (2.5pt);
				\fill  (-3,-1) circle (2.5pt);
				\fill  (-3,0.3) circle (2.5pt);
				\fill  (-3,-0.34) circle (2.5pt);
				\fill  (0,1) circle (2.5pt);
				\fill (0.16,1.43) node {$u_2$};
				\fill (1,1) circle (2.5pt);
				\fill (1.16,1.43) node {$v_2$};
				\fill  (1,0) circle (2.5pt);
				\fill  (2,1) circle (2.5pt);
				\fill  (2,0) circle (2.5pt);
				\fill  (2,-1) circle (2.5pt);
				\fill  (0,0.34) circle (2.5pt);
				\fill  (0,-0.38) circle (2.5pt);
				\fill  (2,0) circle (1.5pt);
				\node at (-3.84,-1.5) {$G_2^1=v_1*(u_1\sqcup G_1^1)$};
				\node at (1.16,-1.5) {$G_2^2=v_2*(u_2\sqcup G_1^2)$};
			\end{scriptsize}
		\end{tikzpicture}
		\caption{Pseudo-Gorenstein graphs of regularity $3$.}
	\end{figure}
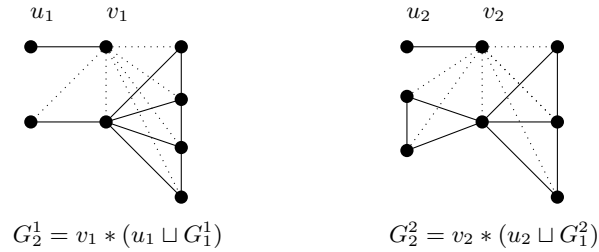
	
\end{example}

Now we characterize pseudo-Gorenstein graph $G$ with $\reg(S/J_G)=3$.
\begin{theorem}\label{thm:pseudo-Gorenstein-with-regularity3}
	Let $G$ be a connected graph on $n$ vertices. Then the following conditions are equivalent:
	\begin{enumerate}
		\item $S/J_G$ is pseudo-Gorenstein with $\reg(S/J_G)=3$;
		\item $G=G_i^1$ or $G=G_i^2$ for some $i\geq 1$, where $G_i^j$'s are defined in Construction \ref{construction-pseudo-Gorenstein}.
	\end{enumerate}
\end{theorem}
\begin{proof}
	$(2)\implies (1)$ Note that if $H=P_1\sqcup P_4$ or $H=P_2\sqcup P_3$, then $\reg(S_H/J_H)=3$. Therefore, by applying repeatedly \cite[Theorem 2.1]{KM-JA}, we have $\reg(S_{G_i^1}/J_{G_i^1})=\reg(S_{G_i^2}/J_{G_i^2})=3$ for $i\geq 1$. Since paths are Gorenstein, by Proposition \ref{cone-level-pseudo}(2), $S_{G_i^1}/J_{G_i^1}$ and $S_{G_i^2}/J_{G_i^2}$ are pseudo-Gorenstein for all $i\geq 1$. 
	
	$(1)\implies (2)$ Since $\reg(S/J_G)=3$, it follows from \cite[Corollary 2.3]{MM} that $G$ is a $P_5$-free graph. Let $X$ be any minimum connected dominating set of $G$. Thus by the proof of Theorem \ref{p_5-free-conjecture}, $G_X$ is isomorphic to $K_m$ for some $m\geq 1$.	First, we show that $m=1$ that is, $G$ is a cone. Let $m\geq 2$. Then by the proof of Theorem \ref{p_5-free-conjecture}, there exists a cut vertex $v\in X$ such that $G\setminus v,G_v$ and $G_v\setminus v$ are accessible. Hence, by Theorem \ref{p_5-free-conjecture}, $G\setminus v,G_v$ and $G_v\setminus v$ are Cohen-Macaulay, which further implies that $\pd(S/((x_v,y_v)+J_{G\setminus v}))=\pd(S/J_{G_v})=n-1$ and  $\pd(S/((x_v,y_v)+J_{G_v\setminus v}))=n$. As $v$ is not a simplicial vertex, by \cite[Lemma 4.8]{oh}, we can write $J_{G}=(J_{G\setminus v}+(x_v,y_v))\cap J_{G_v}$. Thus, we get the following short exact sequence:
	\begin{align}\label{ohtani-ses}
		0\longrightarrow \frac{S}{J_{G}}\longrightarrow 
		\frac{S}{(x_v,y_v)+J_{G \setminus v}} \oplus \frac{S}{J_{{G_v}}}\longrightarrow \frac{S}{(x_v,y_v)+J_{G_v \setminus v}} \longrightarrow 0.
	\end{align}
	Now, considering the long exact sequence of $\Tor $ corresponding to the short exact sequence \eqref{ohtani-ses},  we get
	\begin{align}\label{ohtani-tor}
		0\longrightarrow \Tor_{n}^{S}\left(\frac{S}{(x_v,y_v)+J_{G_v\setminus v}},\K\right)_{n+j} \longrightarrow \Tor_{n-1}^{S}\left( \frac{S}{J_G},\K\right)_{n+j}\longrightarrow \cdots. 
	\end{align}
	Since $S/J_G$ is pseudo-Gorenstein with $\reg(S/J_G)=3$, $\beta_{n-1,n+2}(S/J_G)=1$ is the unique extremal Betti number of $S/J_G$. By the sequence \eqref{ohtani-tor},  we obtain that $\reg(S/((x_v,y_v)+J_{G_v\setminus v}))\in \{1,2\}$.
	Since $G$ has 2 cut vertices, $G_v\setminus v$ is not a complete graph, and hence the regularity must be greater than $1$. Moreover, $\beta_{n,n+2}(S/((x_v,y_v)+J_{G_v\setminus v}))=1$ is the unique extremal Betti number. Therefore by Theorem \ref{CM-reg=2}, $G_v\setminus v=H_i$, with $H_i$ that are defined in Construction \ref{construction-CM-2}. It follows from \cite[Proposition 3.2]{RM-Ext} that $\hat{\beta}(S/((x_v,y_v)+J_{G_v\setminus v}))=\beta_{n,n+2}(S/((x_v,y_v)+J_{G_v\setminus v}))>1$ and we arrive at a contradiction.
	
	Therefore, $m=1$. So $G$ contains only one cut vertex, and hence by \cite[Lemma 4.9]{DAV2}, $G$ is a cone, and $G$ can be written as $G=v*(H_1\sqcup H_2)$, where $G\setminus v=H_1\sqcup H_2$. Let $H=H_1\sqcup H_2$. Then by \cite[Theorem 2.1]{KM-JA}, $\reg(S_H/J_H)=\reg(S_{H_1}/J_{H_1})+\reg(S_{H_2}/J_{H_2})=3$. We assume that $\reg(S_{H_1}/J_{H_1})\leq \reg(S_{H_2}/J_{H_2})$. By Theorem \ref{cone-level-pseudo}, $S_{H_i}/J_{H_i}$ is pseudo-Gorenstein for $i=1,2$. If $\reg(S_{H_1}/J_{H_1})=1$ and $\reg(S_{H_2}/J_{H_2})=2$, then by Remark \ref{rmk:reg<=2}, $H_1=P_2$ and $H_2=P_3$. Hence, $G=G_1^2$. Suppose now $\reg(S_{H_1}/J_{H_1})=0$ and $\reg(S_{H_2}/J_{H_2})=3$. This implies that $H_1$ is an isolated vertex and $|V(H_2)|\geq 4$. Thus, $n\geq 6$. We now prove the assertion by induction on $n$. If $|V(H_2)|=4$, then by \cite[Theorem 3.2]{KM-JCTA}, $H_2=P_4$, and in this case $G=G_1^1$. Let $n>7$. Then $|V(H_2)|>4$. Hence by the inductive hypothesis, we can write $H_2=G_i^j$ for some $i,j$ with $|V(H_2)|\geq 6$. Therefore, $G=G_{i+1}^j$ for $i\geq 1$.
\end{proof}

\section{Bipartite Graph}\label{sec:bipartite}
In this section, we characterize bipartite graphs $G$ such that $S/J_G$ is level or pseudo-Gorenstein. Let $G$ be a connected graph. The complete characterization for Cohen-Macaulay bipartite graph is given in \cite{DAV}. The basic block for this characterization is $F_m$, where $F_m$ is the graph on the vertex set $[2m]$ and with edge set $$E(F_m)= \{\{i,i+1\}:1\leq i\leq 2m-1\}\cup \{\{2i,2j+3\}:1\leq i\leq m-2,i\leq j\leq m-2\} \text{ for } m\geq 1.$$
\begin{definition}\label{def:dot}
	Let $G_i$ be graph with at least one leaf $u_i$ and $v_i\in N_{G_i}(u_i)$ such that $\deg_{G_i}(v_i)\geq 3$ for $i=1,2$. Then we define $G=(G_1,u_1)\circ (G_2,u_2)$ to be the graph obtained by deleting the vertices $u_1,u_2$ and identifying the vertices $v_1,v_2$. For simplicity, we denote $G$ by $G_1\circ G_2$ if $u_1,u_2$ are clear from the context. 
\end{definition}
\begin{theorem}\cite[Theorem 6.1]{DAV}\label{thm:CM-bipartite}
	Let $G$ be a bipartite graph. Then the following conditions are equivalent:
	\begin{enumerate}
		\item $J_G$ is Cohen-Macaulay;
		\item $G$ is decomposable in indecomposable subgraphs $H$, where $H\in \{F_m,F_{m_1}\circ F_{m_2} \circ \cdots \circ F_{m_t}\}$ with $m\geq 1,m_i\geq 3$ for $1\leq i\leq t$ and $t\geq 2$.
	\end{enumerate}
\end{theorem}
We further characterize Cohen-Macaulay bipartite graph $G$ such that $S/J_G$ is level. To achieve this goal, first, we prove that $S/J_G$ is level when $G=F_m$ for $m\geq 1$.

Let $G=F_m$ for $m\geq 1$ and define $f_{i,j}=x_iy_j-x_jy_i$ for $\{i,j\}\in E(G)$ and $i<j$. Then $J_G$ is the ideal
\begin{align*}
	(f_{i,i+1}:1\leq i\leq 2m-1)+(f_{i,j}: i \text{ is even, $j$ is odd, }2\leq i\leq 2m-4,i+3\leq j\leq 2m-1).
\end{align*}
To prove that $S/J_G$ is level we prove that $S/\ini(J_G)$ is level, too.  For this aim we find the generators of  $\ini(J_G)$ with respect to the lexicographic order $<$ induced by $x_1>\cdots>x_n>y_1>\cdots>y_n$. Clearly, it can be seen that the set of degree $2$ elements of $\ini(J_G)$ is
\begin{align*}
	\{x_iy_{i+1}:1\leq i\leq 2m-1\}\cup \{x_iy_j:i \text{ is even, $j$ is odd, }2\leq i\leq 2m-4,i+3\leq j\leq 2m-1\}.
\end{align*}
We recall the notion of admissible path, introduced in  \cite{HHHKR} in order to compute Gr\"obner bases of binomial edge ideals. A path $\pi=(i,i_1,\ldots,i_{r-1},j)$ in a graph $G$ is called {\em admissible}, if
\begin{enumerate}
	\item $i_k\neq i_\ell$ for $k\neq \ell$;
	\item for each $k=1,\ldots,r-1$ one has either $i_k<i$ or $i_k>j$;
	\item for any proper subset $\{j_1,\ldots,j_s\}$ of $\{i_1,\ldots,i_{r-1}\}$, the sequence $(i,j_1,\ldots,j_s,j)$ is not a path.
\end{enumerate}
Given an admissible path $\pi=(i,i_1,\ldots,i_{r-1},j)$ from $i$ to $j$ with $i<j$ we associate the monomial  $u_\pi=(\prod_{i_k>j}x_{i_k})(\prod_{i_\ell<i}y_{i_\ell})$. In \cite{HHHKR}, it is shown that
\[
\ini(J_G)=(x_iy_ju_\pi : \pi \mbox{ is an admissible path}).
\]

To find the elements of degree greater than $2$ in the initial ideal of $J_G$, we describe all admissible paths in $G$. First, we start with a crucial observation about the labeling of $G$.
\begin{remark}\label{remark-F_m}
	Let $G=F_m$. Then it can be noted that if $i$ is even, then $\{i,i'\}\in E(G)$ if and only if $i'$ is odd and either $i'=i-1$ or $i'>i$.
\end{remark}
\begin{lemma}
	Let $G=F_m$ with $m\geq 3$. Then $G$ has no admissible path of length greater than $3$.
\end{lemma}
\begin{proof}
	Let $\bar{P}=(i,i_1,\dots,i_r,j)$ be an admissible path of length $r+1\geq 4$ such that $i<j$ and $\{i,j\}\notin E(G)$. Then for every $1\leq k\leq r$, either $i_k<i$ or $i_k>j$. Since $\bar{P}$ is admissible, $i_k\neq i_l$ for $k\neq l$. We show that there is $\{j_1,\dots,j_s\}\subsetneq \{i_1,\dots,i_r\}$ such that $(i,j_1,\dots,j_s,j)$ is a path in $G$. There are four cases: \noindent \\
	\emph{Case I.} Assume that both $i,j$ are even. Thus $i_r<i$ or $i_r>j$. If $i_r<i<j$, then by Remark \ref{remark-F_m}, $i_r=j-1<i<j$ which is absurd. Therefore, $i_r>j>i$, and so, $\{i_r,i\}\in E(G)$. This implies that $(i,i_r,j)$ is a path in $G$.
	\noindent \\
	\emph{Case II.} If both $i,j$ are odd, then we look $i_1$. If $i_1>j>i$, then by Remark \ref{remark-F_m}, $i_1=i+1>j>i$ which is absurd. So, $i_1<i<j$ which further implies that $\{i_1,j\}\in E(G)$. Therefore, $(i,i_1,j)$ is a path.
	\noindent \\
	\emph{Case III.} Assume that $i$ is even and $j$ is odd. Then by Remark \ref{remark-F_m}, $\{i,j\}\in E(G)$ and hence, $\bar{P}$ is not an admissible path in $G$.
	\noindent \\
	\emph{Case IV.} Now we deal with the remaining case when $i$ is odd and $j$ is even. Let $A=\{i_2,i_4,\dots,i_{r-2}\}$ and $B=\{i_3,i_5,\dots,i_{r-1}\}$. In this case, $r\geq 4$ and so $A\neq \emptyset\neq B$. Assume that $i_k<i$ for all $i_k\in A$ and $i_l>j$ for all $i_l\in B$ which says that in particular, $i_3>j>i>i_2$. As $i_2$ is odd and $\{i_2,i_3\}\in E(G)$, by Remark \ref{remark-F_m}, we have $i_3=i_2+1$ or $i_3<i_2$ which contradicts the inequality. Therefore, there exist $i_k\in A,i_l\in B$ such that $i_k>j$ or $i_l<i$. This yields $\{i_k,j\}\in E(G)$ or $\{i,i_l\}\in E(G)$. Thus we get $(i,i_1,\dots,i_k,j)$ is a path or $(i,i_l,\dots,i_r,j)$ is a path.
	Hence, $G$ does not contain any admissible path of length $\geq 4$.
\end{proof}

Now we are ready to describe all admissible paths in $G$.
\begin{proposition}
	Let $G=F_m$ with $m\geq 3$. Then every admissible path must be of length less than $4$. Moreover, we have:
	\begin{enumerate}
		\item The admissible paths of length $2$ are  $(i,i_1,j)$ such that if $i$ is even (resp. odd), then $j$ is even (resp. odd) and $i_1>j>i$ (resp. $i_1<i<j$).
		\item  The admissible paths of length $3$ are $(i,i_1,i_2,j)$ such that $i$ is odd, $i_1<i$, $i_2>j$ and $i+2<j$.
	\end{enumerate}
\end{proposition}
\begin{proof}
	(1) follows directly from Remark \ref{remark-F_m}. (2) Let  $\bar{P}=(i,i_1,i_2,j)$ be an admissible path of length $3$ with $i<j$. If $i$ is even, then $j$ is odd and so $\{i,j\}\in E(G)$. Therefore, if $\bar{P}$ is an admissible path in $G$, then $\bar{P}$ must start with an odd vertex $i$. Then $i_1$ is even and by Remark \ref{remark-F_m}, $i_1<i$. If $i_2<i$, then $i_2<j$ and so by Remark \ref{remark-F_m}, $i_2=j-1$ which contradicts the fact that $j-1<i<j$. Thus, $i_2>j$. Since $j$ is even and $\{i,j\}\notin E(G)$, $j>i+2$.
\end{proof}

As a Corollary, we find a minimal generating set of the initial ideal of $F_m$.
\begin{corollary}\label{initial-F_m}
	Let $G=F_m$ with $m\geq 3$. Then a minimal generating set of $\ini(J_{F_m})$ consists of the following elements:
	\begin{enumerate}
		\item $\{x_iy_{i+1}: 1\leq i\leq 2m-1\}\cup \{x_iy_j:i\text{ is even, $j$ is odd, }2\leq i\leq 2m-4, \text{ and } i+3\leq j\leq 2m-1\},$
		\item $\{x_ix_{i_1}y_j:i,j \text{ are even, $i_1$ is odd, }2\leq i\leq 2m-4,i+2\leq j\leq 2m-2,i+3\leq i_1\leq 2m-1 \text{ and } i<j<i_1\} \cup
		\{x_iy_{i_1}y_j:i,j \text{ are odd, $i_1$ is even, }3\leq i\leq 2m-3,2\leq i_1\leq i-1,i+2\leq j\leq 2m-1 \text{ and } i_1<i<j\}$, and
		\item $\{x_ix_{i_2}y_{i_1}y_j: i,i_2 \text{ are odd, $i_1,j$ are even, }3\leq i\leq 2m-5,2\leq i_1\leq i-1,i+3\leq j\leq 2m-2,j+1\leq i_2\leq 2m-1 \text{ and } i_1<i<j<i_2\}.$
	\end{enumerate}
\end{corollary} 


\begin{theorem}\label{thm:level-F_m}
	Let $G=F_m$ with $m\geq 1.$ Then $S/\ini(J_G)$ is a level ring. In particular, $S/J_G$ is a level ring.
\end{theorem}
\begin{proof}
	If $m=1,2$, then $G=P_{2m}$. In this case $\ini(J_G)=(x_iy_{i+1}:1\leq i\leq 2m-1)$ is complete intersection. Therefore, $J_G$ is also complete intersection and hence, $S/J_G$ is level. Let $G=F_m$ with $m\geq 3$ and $R=S/\ini(J_G)$. Let $I=(x_{2m},y_1,x_i-y_{i+1}:1\leq i\leq 2m-1)$. Then we see that
	$$\frac{R}{IR}\simeq \frac{S}{IS+\ini(J_G)}\simeq \frac{S'}{I'},$$
	where $S'=\K[x_1,\dots,x_{2m-1}]$ and $I'\subseteq S'$ is the ideal generated by replacing $y_{i+1}$ by $x_i$ for $1\leq i\leq 2m-1$ in the generators of $\ini(J_G)$. Since $\ini(J_G)$ contains the elements of the form $x_iy_{i+1}$ for $1\leq i\leq 2m-1$, the set $\{x_1^2,\dots,x_{2m-1}^2\}\subseteq I'$. Therefore, $I'$ is a $(x_1,\dots,x_{2m-1})$-primary ideal in $S'$ and hence, $\ell(R/IR)=\ell(S'/I')<\infty$ which further implies that $x_{2m},y_1,x_1-y_2,\dots,x_{2m-1}-y_{2m}$ is a homogeneous system of parameters of $R$. Let $R'=S'/I'$. \\
	\noindent
	\emph{Claim:} All the elements in $\soc(R')$ are of degree $3$.
	\\ \noindent
	\emph{Proof of the claim:}
	One generating set of $I'$ is obtained by replacing $y_{i+1}=x_i$ for $1\leq i\leq 2m-1$ in the generating set of $\ini(J_G)$ as defined in Corollary \ref{initial-F_m}. Therefore, we obtain
	\begin{enumerate}[(i)]
		\item $\{x_i^2: 1\leq i\leq 2m-1\}\cup \{x_ix_j:i,j\text{ are even, }2\leq i\leq 2m-4, \ 4\leq j\leq 2m-2 \text{ and } i+1<j\},$
		\item $\{x_ix_jx_{i_1}:i \text{ is even, $j,i_1$ are odd, }2\leq i\leq 2m-4, \ 3\leq j\leq 2m-3, \ 5\leq i_1\leq 2m-1 \text{ and }i<j<i_1-1\} \cup
		\{x_{i_1}x_ix_j:i_1,i \text{ are odd, $j$ is even, }1\leq i_1\leq 2m-1, \ 3\leq i\leq 2m-3, \ 4\leq j\leq 2m-2, \text{ and } i_1+1<i<j\}$, and
		\item $\{x_{i_1}x_ix_jx_{i_2}: i_1,i,i_2,j \text{ are odd, }1\leq i_1\leq 2m-7, \ 3\leq i\leq 2m-5, \ 5\leq j\leq 2m-3, \ 7\leq i_2\leq 2m-1 \text{ and }i_1+1<i<j-1<i_2-2\}.$
	\end{enumerate}
	We want to prove the $\soc(R')$ is generated in degree $3$ and hence, it is level. Clearly, $x_1x_2x_3\notin I'$, and if $i$ is even with $i>2$, then $(x_1x_2x_3)x_i=x_1x_3(x_2x_i)\in I'$ by (i). If $i$ $(>3)$ is odd, then $(x_1x_2x_3)x_i=x_1(x_2x_3x_i)\in I'$ by (ii). Also, $(x_1x_2x_3)x_i\in I'$ for $i=1,2,3$. Thus, $x_1x_2x_3\in \soc
	(R')$, and hence $\soc(R')$ contains a non-zero element of degree $3$.
	
	We observe that all degree $1$ elements are not in $\soc(R')$. In fact, $x_1x_i\notin I'$ for all $2\leq i\leq 2m-1$ implies that $x_i\notin \soc(R')$ for all $2\leq i\leq 2m-1$ and the same holds for $x_1$, too.
	
	Now we focus on degree $2$ elements. First, we observe from the description of (i), that degree $2$ elements not in $I'$ belong to the set
	$A=A_1\cup A_2\cup A_3,$ where $A_1=\{x_ix_j: i \text{ is odd, $j$ is odd, }i<j\}$, $A_2=\{x_ix_j: i \text{ is odd, $j$ is even and }i<j\},$ and $A_3=\{x_ix_j: i \text{ is even, $j$ is odd and }i<j\}$. To show that $x_ix_j\notin \soc(R')$ for $x_ix_j\in A$, it is enough to find an element $x_k$ for $1\leq k\leq 2m-1$ such that $x_ix_jx_k\notin I'$. If $i,j,k$ are all odd, by the generators in (i) and (ii) then $x_ix_jx_k\notin I'$, which further implies that $x_ix_j\notin \soc(R')$ for $x_ix_j\in A_1$. Let $x_ix_j\in A_2$. Then by the generators in (ii), $(x_ix_j)x_{2m-1}\notin I'$, and so $x_ix_j\notin \soc(R')$. Also it can be noted that $x_1(x_ix_j)\notin I'$ for $x_ix_j\in A_3$. Therefore, $x_ix_j\notin \soc(R')$ for $x_ix_j\in A$, and hence $\soc(R')$ does not contain any degree $2$ elements.
	
	Now we show that $I'$ contains all the elements of degree greater than $3$ of $S'$. To this aim it is enough to show that $x_ix_jx_kx_l\in I'$ for any $1\leq i<j<k<l\leq 2m-1$. If at least two elements of $\{i,j,k,l\}$ are even, then by (i), $x_ix_jx_kx_l\in I'$. If $i,j,k,l$ are all odd, then by (iii),  $x_ix_jx_kx_l\in I'$. Suppose exactly one in $\{i,j,k,l\}$ is even. We choose three elements containing the even one from $\{i,j,k,l\}$ such that the even element is either the minimum or the maximum. In both cases, it follows from (ii) that $x_ix_jx_kx_l\in I'$. Thus, the claim follows.
	
	Therefore, it follows from \cite[Chapter III, Proposition 3.2]{Stanley} that $S/\ini(J_G)$ is a level ring, and hence, $S/J_G$ is a level ring, too.
\end{proof}

Now we characterize Cohen-Macaulay bipartite graphs that are level. First we show that for an indecomposable Cohen-Macaulay bipartite graph $G$ on $n$ vertices, $\beta_{n-1,n-1+d}(S/J_G)\neq 0$, where $d=d(G)$.
\begin{remark}\label{extremal-F_m}
	If $G=F_m$ or $G=F_m^{W,2}$, then by \cite[Lemma 3.2, Proposition 3.3]{DAV} and \cite[Theorem 3.4, Proposition 4.1]{JA1}, $S/J_G$ is Cohen-Macaulay and $\reg(S/J_G)=3$ respectively. Therefore, $\beta_{n-1,n-1+3}(S/J_G)$ is the unique extremal Betti number of $S/J_G$. Note that $d(G)=3$.
\end{remark}

\begin{lemma}\label{lemm:non-zero-betti-number}
	Let $F=F_{m}$ or $F=F_m^{W,2}$, a fan graph, with $W=W_1\sqcup W_2\subseteq [m]$ and $|W_1|\geq 2$ for some $m\geq 3$. Let $G=F_{m_1}\circ F_{m_2}\circ \cdots \circ F_{m_t}\circ F$ on $n$ vertices for $m_i\geq 3$. Then $\beta_{n-1,n-1+d(G)}(S/J_G)\neq 0$.
\end{lemma}
\begin{proof}
	Let $V(F_{m_1}\circ F_{m_2}\circ \cdots \circ F_{m_t})\cap V(F)=\{v\}$ and $f_t,f$ be the pendant vertices which are removed from $F_{m_1}\circ F_{m_2}\circ \cdots \circ F_{m_t},F$ respectively.  We proceed by induction on $t$. Observe that the diameter of $G$ is $d(G)=t+3$. Suppose first $t=1$. Then $G=F_{m_1}\circ F_m$ or $G=F_{m_1}\circ F_m^{W,2}$. 
	Let $K$ and $K'$ denote the complete graph with the vertex set $N_G[v]$ and $N_G(v)$ respectively. Then $G\setminus v=(F_{m_1}\setminus \{v,f_1\})\sqcup (F\setminus \{v,f\})$. Note that $G_v$ (resp. $G_v\setminus v$) is a $2$-pure fan graph and obtained by adding fan to $K$ (resp. $K'$) on $W'=(N_{F_{m_1}}(v)\setminus \{f_1\})\sqcup (N_{F}(v)\setminus \{f\})$ if $F=F_m$ or $W''=(N_{F_{m_1}}(v)\setminus \{f_1\})\sqcup (W\setminus W_1)$ if $F=F_m^{W,2}$. It follows from the proof of \cite[Theorem 4.9]{DAV} that $G\setminus v, G_v,G_v\setminus v$ are Cohen-Macaulay graphs. Therefore, $\pd(S/((x_v,y_v)+J_{G\setminus v}))=\pd(S/J_{G_v})=n-1$ and  $\pd(S/((x_v,y_v)+J_{G_v\setminus v}))=n$. 
	By Remark \ref{extremal-F_m}, $\beta_{n,n+3}(S/((x_v,y_v)+J_{G_v\setminus v}))\neq 0$ , and hence it follows from the long exact sequence \eqref{ohtani-tor} for $j=3$ that $\beta_{n-1,n+3}(S/J_G)\neq 0$. 
	
	Assume now $t\geq 2$. Then by the proof of \cite[Theorem 4.9]{DAV}, $G\setminus v$, $G_v$ and $G_v\setminus v$ are Cohen-Macaulay graphs. Therefore, $\pd(S/((x_v,y_v)+J_{G\setminus v}))=n-1$ and $\pd(S/J_{G_v})=n-1$. Also, $G_v=F_{m_1}\circ F_{m_2}\circ\cdots \circ F_{m_{t-1}}\circ F'$, where $F'$ is a fan graph as described in the paragraph above. So $G_v$ and $G_v\setminus v$ satisfy the inductive hypothesis. Hence, $\beta_{n,n+d(G_v\setminus v)}(S/((x_v,y_v)+J_{G_v\setminus v}))\neq 0$. It can be noted that $d(G_v\setminus v)=t+2$. Therefore, considering $j=t+2$ in \eqref{ohtani-tor}, we get $\beta_{n-1,n-1+t+3}(S/J_G)\neq 0$.
\end{proof}

\begin{theorem}\label{thm:level-bipartite}
	Let $G$ be a bipartite graph. Then the following conditions are equivalent:
	\begin{enumerate}
		\item $S/J_G$ is level;
		\item $G$ is decomposable into indecomposable subgraphs $F_m$ for $m\geq 1$.
	\end{enumerate}
\end{theorem}
\begin{proof}
	Let $G$ be a Cohen-Macaulay bipartite graph. By Theorem \ref{thm:CM-bipartite} and Proposition \ref{prop:decomposability-level-pseudo}, it is enough to consider that either $G=F_m$ or $G=F_{m_1}\circ F_{m_2}\circ \cdots \circ F_{m_{t}} \text{ for } m\geq 1,m_i\geq 3.$ By Theorem \ref{thm:level-F_m}, $S/J_{G}$ is level when $G=F_m$. If $G=F_{m_1}\circ F_{m_2} \circ \cdots \circ F_{m_t}$ for $m_i\geq 3$, then by Lemma \ref{lemm:non-zero-betti-number}, $\beta_{n-1,n-1+t+3}(S/J_G)\neq 0$. It follows from \cite[Theorem 4.7]{JA1} that $\reg(S/J_G)>t+3$. Therefore, $S/J_G$ is not level. Hence, the assertion follows.
\end{proof}

Now we characterize bipartite graphs which are pseudo-Gorenstein.
\begin{theorem}\label{thm:pseduo-Gorenstein-bipartite}
	Let $G$ be a bipartite graph. Then the following conditions are equivalent:
	\begin{enumerate}
		\item $S/J_G$ is pseudo-Gorenstein;
		\item $G$ is decomposable into indecomposable subgraphs $H$, where $H\in \{F_1,F_3\circ F_3,F_{m_1}\circ F_{m_2}\circ \cdots \circ F_{m_t}\},$ where $m_1=m_t=3$ and $m_i=4$ for $2\leq i\leq t-1$ with $t\geq 3$.
	\end{enumerate}
\end{theorem}
\begin{proof}
	By Theorem \ref{thm:CM-bipartite} and Proposition \ref{prop:decomposability-level-pseudo}, it is enough to consider that either $G=F_m$ or $G=F_{m_1}\circ F_{m_2}\circ \cdots \circ F_{m_{t}} \text{ for } m\geq 1,m_i\geq 3.$
	Suppose $G=F_{m}$ for $m\geq 1$. Then for $m=1,2$, $G=P_{2m}$, and so $S/J_G$ is pseudo-Gorenstein. If $m\geq 3$, then by Theorem \ref{thm:level-F_m}, $S/J_G$ is level and not Gorenstein by \cite{GBEI}. Hence, $S/J_G$ is not pseudo-Gorenstein. Let $G=F_{m_1}\circ F_{m_2}\circ \cdots \circ F_{m_{t}}$ with $m_i\geq 3$ for $1\leq i\leq t$. Now, we proceed by induction on $t$. For $t=2$, $G=F_{m_1}\circ F_{m_2}$. By \cite[Lemma 5.5]{RM-Ext}, the unique extremal Betti number is given by $\hat{\beta}(S/J_G)=\hat{\beta}(S/J_{F_{m_1-1}})\hat{\beta}(S/J_{F_{m_2-1}})$. Therefore, $\hat{\beta}(S/J_G)=1$ if and only if $m_1=m_2=3$. Hence, $G=F_{m_1}\circ F_{m_2}$ is pseudo-Gorenstein if and only if $m_1=m_2=3$. 
	
	Suppose now $t\geq 3$. Then $G=F_{m_1}\circ F_{m_2}\circ \cdots \circ F_{m_t}$ for $m_i\geq 3$. It follows from \cite[Theorem 5.6]{RM-Ext} that $\hat{\beta}(S/J_G)=1$ if and only if $m_t> 3$ and $\hat{\beta}(S/J_{F_{m_1}\circ \cdots \circ F_{m_{t-1}-1}})=\hat{\beta}(S/J_{F_{m_t-1}})=1$. Therefore, our inductive hypothesis implies that $\hat{\beta}(S/J_{F_{m_1}\circ \cdots \circ F_{m_{t-1}-1}})=1$ if and only if $m_1=3=m_{t-1}-1$ and $m_i=4$ for $2\leq i\leq t-2$. Since $\hat{\beta}(S/J_{F_{m_t-1}})=1$, $m_t=3$. Hence, $m_1=m_t=3$ and $m_i=4$ for $2\leq i\leq t-1$.
\end{proof}

\section{Matroid}\label{sec:matroid}
In \cite{LMRR-S_2}, given a monomial order $<$ the simplicial complex $\Delta_{<}$ induced by $\ini_{<}(J_G)=I_{\Delta_{<}}$ has been defined. It is known (see \cite{Stanley}) that if $\Delta_{<}$ is matroid, then $I_{\Delta_{<}}$ is level. In this section, we classify $G$ such that the corresponding simplicial complex $\Delta_{<}$ is matroid, and this implies that $S/J_G$ is level. 

For $T\in \mathcal{M}(G)$, let $G\setminus T=G_1\sqcup G_2\sqcup \cdots \sqcup G_{c(T)}$. Let $|V(G_i)|=m_i$, say $V(G_i)=\{v_1^i,\dots,v_{m_i}^i\}$ for $1\leq i\leq c(T)$. Given $\textbf{v}=\left(v^1_{j_1},\dots,v_{j_{c(T)}}^{c(T)}\right)\in V(G_1)\times \cdots \times V(G_{c(T)})$, define
\[
F(T,\textbf{v})=\bigcup_{i=1}^{c(T)}\{\{y_j:j\leq v^i_{j_i}\}\cup \{x_j:j\geq v^i_{j_i}\}\}.
\]

\begin{theorem}(\cite[Corollary 1]{LMRR-S_2})\label{simplicial-complex}
	Let $G$ be a graph. Then $\ini_{<}(J_G)=I_{\Delta_{<}}$, where 
	\[
	\mathcal{F}(\Delta_{<})=\bigcup_{T\in \mathcal{M}(G)} \{F(T,\textbf{v}):\textbf{v} \in V(G_1)\times \cdots \times V(G_{c(T)})\}.
	\]
\end{theorem}
\vskip 2mm

Let us recall one equivalent definition of matroid:

\begin{definition}\label{matroid-defn}
	A simplicial complex $\Delta$ is a matroid if for all $F,F'\in \mathcal{F}(\Delta)$ and $i\in F$, there exists a $j\in F'$ such that $(F\setminus \{i\})\cup \{j\}\in \mathcal{F}(\Delta)$.
\end{definition}
\begin{theorem}\label{not-matroid}
	Let $G$ be a connected graph on $n$ vertices with $n\geq 2$. Let $\Delta_{<}$ be a simplicial complex such that $\ini_{<}(J_G)=I_{\Delta_{<}}$. Then $\Delta_{<}$ is matroid if and only if $G=P_n$ for some $n$.
\end{theorem}
\begin{proof}
	Without loss of generality after relabeling we may assume that $x_1>x_2>\cdots > x_n >y_1>y_2 >\cdots >y_n$, for a given monomial order $<$. Let $\Delta_{<}$ be matroid. By \cite[Corollary 3.9]{HHHKR}, $\emptyset \in \mathcal{M}(G)$. Therefore, $$F(\emptyset,(i))=\{y_1,\dots,y_i,x_i,\dots,x_n\}\in \mathcal{F}(\Delta_{<}) \text{ for } 1\leq i\leq n.$$
	Let $F=F(\emptyset,(1))=\{y_1,x_1,x_2,\dots,x_n\}$ and $F'=F(\emptyset,(n))=\{y_1,y_2,\dots,y_n,x_n\}$. Then $F\setminus F'=\{x_1,x_2,\dots,x_{n-1}\}$ and $F'\setminus F=\{y_2,y_3,\dots,y_n\}$.
	We have to find all facets satisfying Definition \ref{matroid-defn} for $F$ and $F'$ above defined, by substituting $x_i\in F$ such that $2\leq i\leq n-1$, with $y_j$ such that $2\leq j\leq n$. If we consider the set
	
	$$H=(F\setminus \{x_i\})\cup \{y_i\}=\{y_1,x_1\}\cup \{x_2,\dots,x_{i-1},y_i,x_{i+1},\dots,x_n\},$$ 
	this is not a facet by the definition \ref{matroid-defn}. In fact, since all the indices of $H$ are in $[n]$ then $H$ is a facet related to the cut set $\emptyset$, and because of our assumption on the labeling $y_i$ must be  $x_i$ obtaining a contradiction. Hence,
	\begin{align*}
		H' & =F\setminus \{x_i\}\cup \{y_j\} \\ & = \{\{y_1,x_1\}\cup \{x_j,y_j\} \cup \{x_2,\dots,x_{j-1},x_{j+1},\dots,x_{i-1},x_{i+1},\dots,x_n\}\} \in \mathcal{F}(\Delta_{<})
	\end{align*}
	for some $j\in \{2,\dots,n\}\setminus \{i\}$. Since all the indices of $H'$ are in $[n]\setminus \{i\}$, $H'$ is a facet related to the cut set $\{i\}$ which implies that $\{i\}\in \mathcal{M}(G)$ for $2\leq i\leq n-1$.
	
	\noindent
	\textit{Claim: }If $G$ has $n-2$ cut vertices, then $G$ is a path.
	
	\noindent
	\textit{Proof of the Claim: }
	We prove it by induction on $n$. For $n=3$, either $G=K_3$ or $G=P_3$. In $K_3$, there are no cut vertices, and so $G=P_3$. Assume that the assertion is true for any graph with $<n$ vertices, and $2,3,\dots,n-1$ are cut vertices of $G$. 
	Since $\{2\}\in \mathcal{M}(G)$, we write $G=G_1\cup G_2$ with $V(G_1)\cap V(G_2)=\{2\}$. 
	For $i\neq 2$, it can be observed that if $\{i\}\in \mathcal{M}(G)$, then either $\{i\}\in \mathcal{M}(G_1)$ or $\{i\}\in \mathcal{M}(G_2)$. Let $1,n\in V(G_1)$, then for the graph $G_2$, all the vertices in $V(G_2)$, with at most the exception of the vertex $2$ are cut vertices. Thus, by inductive hypothesis, $G_2$ is a path graph. This implies that there exists a vertex $v\neq 2$ with $v\in V(G_2)$ that is a simplicial vertex, but this is absurd because $v$ is a cut vertex.

	Therefore, we may assume that $1\in V(G_1)$ and $n\in V(G_2)$. Since $G_1$ and $G_2$ both satisfy the inductive hypothesis, $G_1,G_2$ both are path graphs, and the vertex $2$ is a simplicial vertex for both paths. Hence, $G$ is also a path graph. Hence, if $\Delta_{<}$ is a matroid, then $G$ is a path. 
	
	Let $G=P_n$ on $[n]$ with $E(G)=\{\{i,i+1\}:1\leq i\leq n-1\}$. We consider the lexicographic order $<$ induced by $x_1>x_2>\cdots > x_n >y_1>y_2 >\cdots >y_n$. Then $\ini_{<}(J_G)$ is complete intersection. Therefore, by \cite[Chapter III, Proposition 3.1]{Stanley}, $\Delta_{<}$ is matroid.
\end{proof}

\section{Computation of graphs with $n\in\{2,\ldots,12\}$ vertices}\label{sec: computation}

The main aim of this section is to find, using a computational approach, the cardinalities of the set of graphs $G$ with at most 12 vertices such that the corresponding rings $S/J_G$ are level or pseudo-Gorenstein. This computation is based upon the database provided in the paper \cite{LMRR}. Finally, we discuss some interesting examples obtained.

The previous procedure was executed for the graphs whose number of vertices is between $2$ and $12$ that are Cohen-Macaulay. In Table \ref{tab:indec}, we report the number of indecomposable graphs on $n$ vertices.
\begin{table}[H]
	\centering
	\begin{tabular}{|c|c|c|c|c|c|c|c|c|c|c|c|c|}
		\hline
		$n$            &2 &3 &4 &5 &6 &7  &8  &9   &10  &11   &12    &Tot\\
		\hline 
		Cohen-Macaulay &1 &1 &1 &2 &5 &15 &51 &194 &833 &3824 &19343 &24270   \\
		\hline
		Level          &1 &1 &1 &2 &3 &5  &12 &27  &82  &231    &726     & 1091  \\
		\hline
		pseudo-Gorenstein &1 &0 &0 &0 &2 &5  &8 &34  &144  &520    &2303     &3017   \\
		\hline
		
	\end{tabular}
	\caption{Enumeration of indecomposable Cohen-Macaulay, level and pseudo-Gorenstein rings}
	\label{tab:indec}
\end{table}

We observe that the only indecomposable level binomial edge ideal that is pseudo-Gorenstein is the one in the first column that is the the graph $K_2$, that is the Gorenstein one. In fact, any path is decomposable into indecomposable graphs that are $K_2$ (its edges). Hence in every column different from the first one, the set of pseudo-Gorenstein does not intersect with the set of level ones, since path are the only Gorenstein rings in this context. 

Moreover, in the column representing the graphs with $6$ vertices,  the five Cohen-Macaulay binomial edge ideals are either level or pseudo-Gorenstein. These five graphs belong to the ones studied in this paper. They are two pseudo-Gorenstein rings of regularity $3$ related to the graphs described in the  Example \ref{exa:pseudo-gorenstein-cones}, i.e. the graph $G_1^1$ and $G_1^2$. Then there are three level rings. In these cases the related graphs are : the complete graph $K_6$, the bipartite graph $F_3$, that has been proved to be level in Section \ref{sec:bipartite}, and the graph $H_2$ of Example \ref{exa:levelcones} setting $K_{r+1}=K_3$ and $K_{s+1}=K_2$.

In the set of the graphs with $7$ vertices, there are two  interesting graphs that induce pseudo-Gorenstein rings  of regularity $4$.
They are the graphs with $7$ vertices of Fig.4 and Fig.7 of \cite{Rinaldo-Cactus}, and they are not cones.
Moreover, the first one belongs to the set of cactus graph, and the second one belongs to the set of chain of cycles, studied respectively in \cite{Rinaldo-Cactus} and \cite{LMRR-S_2}. In those cases the Cohen-Macaulay property is equivalent to combinatorial properties, namely accessible and strongly-unmixed properties. So the following question naturally arises
\begin{question}
 Is it possible to characterize level and pseudo-Gorenstein binomial edge ideals of cactus (resp. chain of cycles) graphs?	
\end{question}
If this question has a positive answer, because of the characterization of chordal and traceable graphs of \cite{DAV2}, then
\begin{question}
	Is it possible to characterize level and pseudo-Gorenstein binomial edge ideals of chordal (resp. traceable) graphs?	
\end{question}

Finally, we refer readers to \cite{RS} for a complete description of the algorithm that we used, the obtained database, and the code in Macaulay 2.

\bibliographystyle{plain}
\bibliography{Reference}

\end{document}